\documentclass[reqno, 11pt]{amsart}

\usepackage[text={6in,8.5in},centering,letterpaper,dvips]{geometry}
\usepackage{amsmath, amscd,amsthm}%
\usepackage{amsfonts}%
\usepackage{amssymb}%
\usepackage{graphicx}
\usepackage[normalem]{ulem}
\usepackage[small]{diagrams}
\usepackage{pstricks}
\usepackage{newclude}
\usepackage{enumerate}
\usepackage{multido}
\usepackage{hyperref}
\usepackage{tikz-cd}
\usepackage{slashed}

\newcommand{\nocontentsline}[3]{}
\newcommand{\tocless}[2]{\bgroup\let\addcontentsline=\nocontentsline#1{#2}\egroup}

\newtheorem{theorem}{Theorem}
\theoremstyle{plain}

\newtheorem{corollary}{Corollary}

\newtheorem{lemma}{Lemma}

\newtheorem{proposition}{Proposition}

\numberwithin{equation}{section}

\theoremstyle{definition}
\newtheorem{definition}{Definition}

\newtheorem*{acknowledgement}{Acknowledgements}

\theoremstyle{remark}
\newtheorem*{rem}{Remark}
\newtheorem*{note}{Note}

\newcommand{\CSD}{\mathcal{L}}

\newcommand{\quadr}[1]{\rho^{-1} \left( #1 #1^* \right)_0}

\newcommand{\cindex}[1]{\ensuremath{\mathcal{I}(#1)} }

\newcommand{\Inv}[1]{\ensuremath{\operatorname{Inv}(#1)}}

\newcommand{\Inter}[1]{\ensuremath{\operatorname{Int}(#1)}}
\newcommand{\Image}[1]{\ensuremath{\operatorname{Im}(#1)}}
\newcommand{\Ker}[1]{\ensuremath{\operatorname{Ker}(#1)}}
\newcommand{\Coker}[1]{\ensuremath{\operatorname{Coker}(#1)}}

\DeclareFontFamily{U}{mathx}{\hyphenchar\font45}
\DeclareFontShape{U}{mathx}{m}{n}{
      <5> <6> <7> <8> <9> <10>
      <10.95> <12> <14.4> <17.28> <20.74> <24.88>
      mathx10
      }{}
\DeclareSymbolFont{mathx}{U}{mathx}{m}{n}
\DeclareFontSubstitution{U}{mathx}{m}{n}
\DeclareMathAccent{\widecheck}{0}{mathx}{"71}
\DeclareMathAccent{\wideparen}{0}{mathx}{"75}
\DeclareMathAccent{\widebar}{0}{mathx}{"73}

\begin{document}

\title{A new gauge slice for the relative Bauer-Furuta invariants}     %% For first page.
\author{Tirasan Khandhawit}              %% Include coauthors if any.
\date{} 
\address{Kavli Institute for the Physics and Mathematics of the Universe (WPI), Todai Institutes for Advanced
Study, The University of Tokyo, 5-1-5 Kashiwa-No-Ha, Kashiwa, Chiba 277-8583, Japan}
\email{tirasan.khandhawit@ipmu.jp}
\urladdr{} 

\begin{abstract}
In this paper, we study Manolescu's construction of the relative Bauer-Furuta invariants arising from the Seiberg-Witten equations on 4-manifolds with boundary. The main goal of this paper is to introduce a new gauge fixing condition in order to apply the finite dimensional approximation technique. We also hope to provide a framework to extend Manolescu's construction to general 4-manifolds.
\end{abstract}

\maketitle
%\tableofcontents

\pagestyle{plain}

\section{Introduction}

Stable homotopy invariants arising from gauge theory have provided many interesting results in low-dimensional topology.
One of the first examples is Furuta's $10/8$-theorem, which provides constraints on intersection forms of smooth 4-manifolds \cite{Furu}. Later, Bauer and Furuta constructed an invariant for a closed 4-manifold as an element in a certain stable cohomotopy group  (\cite{BFII}, \cite{BFI}).

The basic idea of this construction is to consider the Seiberg-Witten map, rather than its moduli space of solutions, and then consider approximated maps between finite dimensional vector spaces  to obtain a stable map between spheres. One useful observation for this construction is that the Seiberg-Witten map can be written as a sum of linear and compact operators. 

In 2003, Manolescu constructed a Floer spectrum for a rational homology 3-sphere \cite{Man1}. Roughly speaking, the construction comes from finite dimensional approximation of the Seiberg-Witten flow on an infinite-dimensional space. This allows one to extend the notion of Bauer-Furuta invariants to 4-manifolds with a rational homology sphere as a boundary. 

Let $X$ be a smooth, compact, connected, oriented 4-manifold with boundary \(\partial X = Y\) and equip $X$ with a spin$^c$ structure whose restriction induces a spin$^c$ structure on $Y$. Conceptually, one can view the construction of the relative Bauer-Furuta invariant as a combination of finite dimensional approximation on both \(X\)  and  \(Y\) using  the Seiberg-Witten map and the restriction map 
\begin{align*}\mathcal{M}(X) \rightarrow \mathcal{B}(Y)\end{align*}
from the moduli space of Seiberg-Witten solutions of $X$ to the quotient configuration space of $Y$ as a boundary term. An important step is to instead consider spaces of configurations with certain gauge fixing condition so that we have a map with Fredholm property between vector spaces. 

The main purpose of this paper is to introduce a new gauge fixing for a 4-manifold with boundary. An advantage of our gauge fixing condition, called the double Coulomb condition, is that the restriction map from the corresponding slice on \(X\) to the Coulomb slice on \(Y\) is linear. In contrast, the restriction map from the previously used Coulomb-Neumann slice on \(X\) to the Coulomb slice on \(Y\) is not linear and its nonlinear part is not compact. In fact, the boundary condition of our double Coulomb condition is motivated by this situation.

In Section~\ref{sec prelim}, we give a definition of the double Coulomb condition and prove its basic properties. In Section~\ref{sec fredholm}, we show that
the double Coulomb slice has several properties analogous to the Coulomb-Neumann slice. In Section~\ref{sec main}, we apply finite dimensional approximation to this slice. At the end, we specialize to the case when $b_1(Y) = 0$ and reproduce Manolescu's construction of the relative Bauer-Furuta invariant, denoted by \(SWF(X)\). \begin{theorem}
When \(b_1 (Y) = 0\), the Seiberg-Witten map with boundary term on the double Coulomb slice gives an $S^1$-equivariant stable homotopy class of maps
\begin{align}
\mathit{SWF}(X) : \Sigma^{-b^+ (X)} Th_{Dir} (X) \rightarrow \mathit{SWF}(Y), 
\label{eq mainth}
\end{align}
where $\mathit{SWF}(Y)$ is the Floer spectrum associated to \(Y\) and $Th_{Dir} (X)$ is the Thom spectrum associated to a family of Dirac operator on $X$ parametrized by the Picard torus of \(X\).  In the special case when $b_1 (X) = 0$, we have
\begin{align}
\mathit{SWF}(X) : \mathbf{S}^{-b^+ (X)\mathbb{R} - \frac{\sigma(X)}{8}\mathbb{C}}  \rightarrow \mathit{SWF}(Y), 
\label{eq mainb1=0}
\end{align}
where \(\mathbf{S}\) is the sphere spectrum.
\label{thm b1=0}
\end{theorem}

In particular, when \(X\) is a cobordism between two 3-manifolds, one can use duality and reinterpret the relative Bauer-Furuta invariant as a morphism between Floer spectra.

\begin{corollary} Suppose that $\partial X = \widebar{Y_1} \coprod Y_2$ and $b_1 (X) = b_1(Y_1) = b_1 (Y_2) = 0$, then we also have an $S^1$-equivariant stable homotopy class of maps
\begin{align}
\mathit{SWF}(X) : \mathit{SWF}(Y_1)  \rightarrow \mathbf{S}^{b^+ (X)\mathbb{R} + \frac{\sigma(X)}{8}\mathbb{C}} \wedge \mathit{SWF}(Y_2). \nonumber
\end{align}
\label{cor cobord}
\end{corollary}
We point out that the construction can be applied directly to give a stable homotopy class of $\mathit{Pin}(2)$-equivariant maps when all manifolds are spin. This \(Pin(2)\)-version of   Corollary~\ref{cor cobord} plays a crucial role in the recent applications of $\mathit{Pin}(2)$-equivariant stable homotopy invariants to low-dimensional topology \cite{Lin14,Man13-2,Man13-1}.

Another goal of the paper is to provide a framework to extend Manolescu's construction to a 4-manifold whose boundary can be any 3-manifold. The case $b_1 (Y) = 1$ was studied by Kronheimer and Manolescu in \cite{ManK}. We also hope that this new slice will help prove other important properties of  the relative Bauer-Furuta invariant.

In Appendix~\ref{App con}, we provide some background in Conley index theory. We also include the result regarding independence of index pairs in the construction (see Proposition~\ref{prop maptocon}), which, we believe, has not appeared before. 

\begin{acknowledgement}  This was part of the author's Ph.D. thesis at Massachusetts Institute of Technology. The author would like to gratefully thank Tom Mrowka for advising and support during the graduate study. The author would also like to thank Mikio Furuta and Ciprian Manolescu for several helpful discussions. This work was supported by World Premier International Research Center Initiative (WPI), MEXT, Japan.
\end{acknowledgement}

%%%%%%%%%%%%%%%%%%%%%%%%%%%%%%%%%%%%%%%%%%%%%%
\section{Preliminaries: The double Coulomb condition}
\label{sec prelim}

In this section, let $M$ be a compact, connected, oriented, Riemannian manifold with boundary $\partial M  = \coprod N_i$. We will describe a variant of the the Hodge decomposition of 1-forms in order to set up an appropriate slice for finite dimensional approximation of the Seiberg-Witten map.

The inclusion $\partial M \rightarrow M$ gives a decomposition of a differential form of $M$ at the boundary to tangential part and normal part
\begin{align}
\omega_{| \partial M} = \mathbf{t}\omega + \mathbf{n}\omega . \nonumber
 \end{align} 

Thus, $\mathbf{t} \omega$ is the restriction of $\omega$ to the boundary. When $\partial M $ has more than one connected components, we also denote by $\mathbf{t}_i$  the restriction to the $i$-th component.

Let $\star$ be the Hodge star and $d^*$ be the codifferential. With this notation, we recall the formula for integration by parts (namely, the Green's formula)
\begin{align}
\int_M \left\langle d \omega , \eta \right\rangle = \int_M \left\langle  \omega , d^* \eta \right\rangle + \int_{\partial M} \mathbf{t} \omega \wedge \star \mathbf{n}\eta \label{eq green}
\end{align}
and the identity $\star (\mathbf{n} \omega) = \mathbf{t} (\star \omega) $.

We now introduce a space of 1-forms with double Coulomb condition.
\begin{definition} We say that a 1-form $\alpha$ satisfies the double Coulomb condition if
\begin{enumerate}
\item $\alpha$ is coclosed ($d^* \alpha = 0$).
\item its restriction to the boundary is coclosed, i.e. $d^* ( \mathbf{t} \alpha ) = 0$ on $\partial M$.
\item For each $i$, the integral $\int_{N_i} \mathbf{t}_i (\star \alpha) = 0$. 
\end{enumerate}
Denote by $\Omega^1_{\mathit{CC}} (M)$ the space of all 1-forms satisfying the double Coulomb condition. \label{def 1-bC}
\end{definition}

When the metric of \(M\) is cylindrical near the boundary, i.e. the neighborhood of the boundary is isometric to \((-\epsilon,0] \times \partial M\), one can decompose a 1-form $\omega$ in the collar neighborhood as
\begin{align}
        \omega = \alpha(t) + \beta(t) + \gamma(t) dt , 
\label{eq 1formcylin}        
\end{align}
where $\alpha(t) , \beta(t) ,$ and $\gamma(t)$ is an exact 1-form, a coclosed 1-form, and a 0-form on $\partial M$ respectively and each of them is time dependent.
One can see that the Neumann condition $\mathbf{n}\omega =0$ simply means $\gamma(0) = 0$, while the  condition $d^* ( \mathbf{t}\omega ) = 0$ means $\alpha(0) = 0$ and the last condition in Definition~\ref{def 1-bC} is a condition on the integral of \(\gamma(0) \).

Hodge theory on manifolds with boundary has been studied by many authors (cf. \cite{HodgeB,GSchwarz}). However, the double Coulomb condition and the following decomposition appear to be new. 

\begin{proposition} Any 1-form $\alpha$ can be written uniquely as a sum $\alpha = \omega + d \xi$ where $\omega \in \Omega^1_\mathit{CC} (M)$ and $\xi$ is a 0-form. In other words, there is an isomorphism
\begin{align}
   \Omega^1 (M) \cong \Omega^1_\mathit{CC} (M) \oplus d \Omega^0 (M) \label{eq CCdecom}
   \end{align}   
   \label{prop 1-decomp}
\end{proposition}

\begin{rem} This decomposition is not orthogonal as opposed to the standard Hodge decomposition $\Omega^1 (M) \cong \Omega^1_\mathit{CN} (M) \oplus d \Omega^0 (M)$, where $\Omega^1_\mathit{CN} (M)$ is a space of coclosed 1-form with vanishing normal component ($\mathbf{n} \alpha = 0$). If the condition $\int_{N_i} \mathbf{t}_i (\star \alpha) = 0$ is omitted, a decomposition exists but not unique with ambiguity of dimension $b_0 (\partial M) - 1$.
\end{rem}

\begin{proof}
Let $\alpha$ be a 1-form. We will first find a function $\xi_1$ such that $\alpha - d\xi_1$ satisfies (i) and (ii) of Definition~\ref{def 1-bC}, that is
\begin{equation}
\begin{aligned}
d^* ( d \xi_1) &= d^* \alpha, \\
d^* ( \mathbf{t} d\xi_1 ) &= d^* ( \mathbf{t} \alpha ).
\end{aligned} \label{eq hodge1}
\end{equation}
Since $d$ and $\mathbf{t}$ commute, we see that $d^* ( \mathbf{t} d\xi_1 ) = \Delta \mathbf{t} \xi_1 $. We can instead consider 
\begin{equation*}
\begin{aligned}
\Delta \xi_1 &= d^* \alpha, \\
\mathbf{t} \xi_1  &= G d^* ( \mathbf{t} \alpha ),
\end{aligned} 
\end{equation*}
where $G$ is a Green's operator of the Laplacian of the boundary $\partial M$. This equation is precisely the Dirichlet problem for the Poisson equation and is uniquely solvable (cf.~\cite{GSchwarz}). As a result, one can always find such $\xi_1$ as claimed.

Next, consider a function $\xi$ which satisfies
\begin{equation}
\begin{aligned}
\Delta \xi &= 0, \\
\mathbf{t}_i \xi  &= c_i ,
\end{aligned} 
\label{eq hodgeker}
\end{equation}
where $c_i$ are constants. This is a homogeneous solution of (\ref{eq hodge1}). Since this equation is also the Dirichlet problem, there is a unique solution $\xi$ for each vector $c =(c_1 , \ldots , c_{b_0 (\partial M)})$. 

Denote by $K$  the space of functions $\xi $ satisfying (\ref{eq hodgeker}) for all possible \(c\).    
%$ = \left\lbrace \xi \in \Omega^0 (X) \: | \: \xi \text{ satisfies (\ref{eq hodgeker}) with} \sum c_i = 0  \right\rbrace  $
This is a vector space of dimension $b_0 (\partial M) $ and $\xi$ is a constant function when the $c_i$'s are all equal. Let us consider a map
$ ev : K \rightarrow \mathbb{R}^{b_0 (\partial M)}$ by assigning the value
$ \int_{N_i} \mathbf{t}_i (\star d \xi) $ to the $i$-th component. When $\eta = d \xi$ and $\omega$ is a nonzero constant function, the Green's formula (\ref{eq green})  implies that
\begin{align}
0 = \int_M \Delta \xi + \sum \int_{N_i} \mathbf{t}_i (\star d \xi). \nonumber
\end{align}
Then, we see that the image of $ev$ is in the hyperplane $H_0 := \left\lbrace (r_i) \in \mathbb{R}^{b_0 (\partial M)}  \: | \: \sum r_i = 0 \right\rbrace $. By plugging $(\omega , \eta ) = (\xi , d\xi)$ in the Green's formula, we find that the kernel of $ev$ consists of the constant functions. Thus, the map $ev$ gives an isomorphism between 
\( K_0 :=
\left\lbrace \xi \in K \: | \: \sum c_i = 0 \right\rbrace \) and the hyperplane \( H_0 \).

Finally, we notice that any coclosed 1-form \(\omega\) satisfies
$ 0 = \sum \int_{N_i} \mathbf{t}_i (\star \omega) $ by pairing $\omega$ with a nonzero constant function in the Green's formula. Note that this makes the  condition (iii) trivial when $\partial M$ has just one component.
 From the previous paragraph, we can now pick $\xi_2 \in K_0$ so that $\int_{N_i} \mathbf{t}_i (\star (\alpha - d \xi_1 - d \xi_2)) = 0$. Hence \(\alpha - d \xi_1 - d \xi_2 \in \Omega^1_\mathit{CC} (M) \). The uniqueness of the decomposition follows from the fact that the kernel of $ev$ consists of the constant functions.

\end{proof}

\begin{note}
For a coclosed 1-form \(\alpha\), the condition $\int_{N_i} \mathbf{t}_i (\star \alpha) = 0$ is equivalent to $\alpha$ being orthogonal to $d \xi$ for all $\xi \in K$. Indeed, we have
\[ \int_M \left\langle d \xi , \alpha \right\rangle = \sum c_i \int_{N_i} \mathbf{t}_i (\star \alpha). \]
\end{note}

%%%%%%%%%%%%%%%%%%%%%%%

\section{The Seiberg-Witten map with boundary terms}
\label{sec fredholm}
%\section{Atiyah-Patodi-Singer boundary-value problem}
From now on, let $X$ be a compact, connected, oriented, Riemannian 4-manifold with nonempty boundary $\partial X =Y$. 
We choose a metric so that a neighborhood of the boundary is isometric to the cylinder $I \times Y$ for some interval $I = (-C ,0]$.
Let $\mathfrak{s}_X$ be a spin$^c$ structure on $X$ and $\mathfrak{s}$ be the induced spin$^c$ structure on $Y$.
Denote by $S_X = S^+ \oplus S^-$ the spinor bundle of $X$ and by $S$  the spinor bundle of $Y$. 

Denote by \(\mathcal{A}_X\)  the space of spin$^c$ connection on \(S_X\) and by \(\Gamma(S^\pm)\)  the space of sections of the spinor bundles and by $\Omega_+^2 (X)$  the space of self-dual 2-forms. The Seiberg-Witten map is given by
\begin{align*}
\mathit{SW}: \mathcal{A}_X \oplus \Gamma(S^+) &\rightarrow i \Omega_+^2 (X) \oplus \Gamma(S^-) \\
(A,\Phi) &\mapsto (\frac{1}{2} F^+_{A^t} - \rho^{-1} ((\Phi \Phi^* )_0), \slashed{D}_{A}^+ \Phi), 
\end{align*}
where  \(F^+_{A^t}\) is the self-dual part of the curvature of the associated connection on the determinant bundle \(\Lambda^2 S^+ \),  \(\slashed{D}_{A}^+\) is the Dirac operator, \((\Phi \Phi^* )_0\) is the trace-free part of the endomorphism \(\Phi \Phi^*\), and \(\rho\) is the Clifford multiplication.    

The gauge group $\mathcal{G} := Map(X , S^1)$ acts on the above spaces so that $SW$ is  $\mathcal{G}$-equivariant. The action is given by $u \cdot A \mapsto A - u^{-1} d u$ on connections and by pointwise multiplication on  spinors whereas the action is trivial on 2-forms.
There is also the gauge subgroup  $\mathcal{G}^\bot := \left\{ e^\xi \; | \; \xi \in C^{\infty}(X; i\mathbb{R}) \text{ and } \int_X \xi = 0 \right\}$, which lies in the connected component of \(\mathcal{G}\).  

With a reference connection \(A_{0}\), the quotient of \(\mathcal{A}_X \oplus \Gamma(S^+)\) by \(\mathcal{G}^\bot\) can be identified with a global slice with the double Coulomb condition
\begin{align}
\mathit{Coul}^\mathit{CC}(X) =  \left\{ (a , \phi ) \in i \Omega^1 (X) \oplus \Gamma (S^+) \; | \; a \in\Omega^1_\mathit{CC} (X)    \right\}. \nonumber
\end{align}
This is a consequence of the decomposition (\ref{eq CCdecom}) from Proposition~\ref{prop 1-decomp}. The Seiberg-Witten map then becomes
\begin{align*}
SW: \mathit{Coul}^{CC}(X)  &\rightarrow  i \Omega_+^2 (X) \oplus \Gamma(S^-)  \\
(a,\phi) &\mapsto (d^+ a  - \rho^{-1} ((\phi \phi^* )_0)+ \frac{1}{2} F^+_{A_0^t} \, , \, \slashed{D}_{A_{0}}^+ \phi + \rho(a)\phi \, ) 
\end{align*} 
and we can write \(SW = \hat{D} + \hat Q\) where \(\hat D = (d^+ , \slashed{D}_{A_{0} }^+ )\) and \(\hat Q\) is the sum of a quadratic map and the constant term \(\frac{1}{2} F^+_{A_0^t}\).

On the 3-manifold side, we also have the Coulomb slice 
\begin{equation*}       \mathit{Coul}(Y) = \left\{ \left( b  , \psi  \right) \in i \Omega^1 (Y) \oplus \Gamma (S) \; | \; d^* b = 0 \right\}
\end{equation*}
which arises from the quotient of a configuration space by a gauge subgroup. For \(a \in \Omega^1_{CC} (X)\), its restriction to the boundary is already coclosed, so that the restriction  induces a map between the slices
\begin{align*}r : \mathit{Coul}^{CC}(X) \rightarrow \mathit{Coul}(Y). \end{align*}

This gives a Seiberg-Witten map with boundary terms
\begin{align*}
 SW \oplus r: \mathit{Coul}^{\mathit{CC}}(X)  &\rightarrow \left( i \Omega_+^2 (X) \oplus \Gamma(S^-) \right) \oplus \mathit{Coul}(Y).  
\end{align*}
As usual, we will extend the above maps to maps between appropriate Sobolev spaces. For a fixed real number\footnote{ This $k$ corresponds to a half-integer \(k+1/2 \) in \cite{Man1}.} \(k>3\), we consider the \(L^{2}_{k+1}\) completion of the domain of \(SW\), the \(L_k^2\) completion of the codomain of \(SW\), and the \(L^2_{k + 1/2}\) completion of \(Coul(Y)\) so that \(\hat Q\) is compact and the restriction map \(r\) is bounded. 　  

However, the linear part \(\hat{D} \oplus r\) is not Fredholm. To obtain a Fredholm map, we need to consider the above operator with spectral boundary condition as in the Atiyah-Patodi-Singer boundary-value problems. On the boundary 3-manifold \(Y\), we have a first-order self-adjoint elliptic operator \(D\) acting on the
Coulomb slice
 \begin{align}
        D : i \Ker{d^*} \oplus \Gamma (S) &\rightarrow i \Ker{d^*} \oplus \Gamma (S) \label{eq linearmap3d} \\ 
        (b,\psi ) &\mapsto (* d b  , \slashed{D}_{B_0} \psi ), \nonumber
\end{align}
where the connection \(B_0\) is the restriction of \(A_0\). Denote by $H_0^-$  its nonpositive eigenspace and by $\Pi_0^-$  the projection onto $H_0^-$. 
We now consider a map of the form
\begin{align}
\hat{D} \oplus (\Pi^-_0 \circ r) : \mathit{Coul}^{\mathit{CC}} (X) \rightarrow \left( i \Omega_+^2 (X) \oplus \Gamma(S^-) \right)  \oplus H^-_0 . \label{eq map-coul}
\end{align}
We will show that this map, extended to the Sobolev completion, is Fredholm with an elliptic estimate. The proof resembles that of Proposition~17.3.2 from \cite{Mono}.

\begin{proposition} The map $\hat{D} \oplus (\Pi^-_0 \circ r)$ in (\ref{eq map-coul}) is Fredholm and its index is equal to \(2\operatorname{Ind}_{\mathbb{C}} (\slashed{D}_{A_0}^+ ) +b_1(X) - b^+ (X) - b_1(Y) \). In addition, we have an estimate
\begin{align}
        \left\|  x \right\|_{L^2_{k+ 1}} \leq C \left( \left\|\hat{D} x \right\|_{L^2_{k}} + \left\| (\Pi^-_0 \circ r) x \right\|_{L^2_{k+ 1/2}} + \left\| x \right\|_{L^2} \right).
\label{eq ellipest}
\end{align}
\label{prop bFred}
\end{proposition}

\begin{proof} The main idea is to apply the Atiyah-Patodi-Singer boundary-value problem (cf.~\cite{APSI}) to the extended operator coming from the gauge fixing condition. Then, we will compare projections onto different semi-infinite subspaces in the boundary terms. One subspace arises from a spectral boundary condition of the extended operator while another is a sum of an eigenspace on \(\mathit{Coul}(Y)\) and a subspace characterizing the double Coulomb condition.

We will also make use of the following observation: suppose that an operator $\widetilde{D}$  arises from two operators $D_1 : V \rightarrow W_1$ and $D_2 : V \rightarrow W_2$, in the sense that 
\begin{align*}
        \widetilde{D}  = ( D_1 , D_2 ) : V \rightarrow W_1 \oplus W_2 .
\end{align*}
It is not hard to check that
\begin{equation}
\begin{aligned}
        \Ker{\widetilde{D} } &= \Ker{{D_1}|_{\Ker{D_2}}} = \Ker{{D_2}|_{\Ker{D_1}}}  \\
        \Coker{\widetilde{D} }  &= \Coker{D_1} \oplus \Coker{D_2|_{\Ker{D_1}}} \\ &= \Coker{D_1|_{\Ker{D_2}}} \oplus \Coker{D_2} .
\end{aligned}   
\label{eq sumFredholm}
\end{equation}
Consequently, the map $D_1|_{\Ker{D_2}} : \Ker{D_2}  \rightarrow W_1$ is Fredholm if \(\widetilde{D} \) is Fredholm. 

Let us start by considering an elliptic operator $\widetilde{D}$ given by
\begin{align*}
\widetilde{D} : i \Omega^1 (X) \oplus \Gamma (S^+) &\rightarrow i \Omega_+^2 (X) \oplus \Gamma(S^-) \oplus i \Omega^0 (X) \\
(a,\phi) &\mapsto (d^+ a , D_{A_0}^+ \phi , d^* a). 
\end{align*}
This is the map  $\hat D$ together with the summand $d^*$ for gauge fixing. Then, we write $\widetilde{D} = D_0 + K$, where $K$ extends to an operator of order $0$ and $D_0$ has the form
\begin{align*}
        D_0 = \frac{d}{dt} + \widetilde{L}, 
\end{align*}
in the collar neighborhood (up to isomorphisms) and the operator $\widetilde{L}$ is a first-order, self-adjoint elliptic operator given by
\begin{align*}
        \widetilde{L} : i \Omega^1 (Y) \oplus \Gamma (S) \oplus i \Omega^0 (Y) &\rightarrow i \Omega^1 (Y) \oplus \Gamma(S) \oplus i \Omega^0 (Y) \\ 
        (b,\psi , c) &\mapsto (* d b - dc , \slashed{D}_{B_0} \psi , - d^* b).
\end{align*}

%Similarly, the first two summands of $\widetilde{L}$ is a linear part of the 3-dimensional Seiberg-Witten map and the last term is gauge fixing. 
Using the Hodge decomposition, the restriction of $\widetilde{L}$ to $i \Omega^1 (Y) \oplus i \Omega^0 (Y) = i \Image{d} \oplus i\Ker{d^*} \oplus i \Omega^0 (Y)$ can be written as a block
\begin{align*} \begin{bmatrix} 0 & 0 & -d \\ 0 & *d & 0 \\ -d^* &0 &0 \end{bmatrix}.
        \end{align*}
One can also rearrange and view the domain (and the codomain) of $\widetilde{L}$ as $\mathit{Coul}(Y) \oplus i \Image{d} \oplus i \Omega^0 (Y)$, so that $\widetilde{L} = D \oplus L_1$ (as in (\ref{eq linearmap3d})) and  $L_1$ is an operator on $i \Image{d} \oplus i \Omega^0 (Y)$ with a block form
\begin{align*} \begin{bmatrix} 0 & -d \\ -d^* &0  \end{bmatrix}.
        \end{align*}     
                
We now apply the Atiyah-Patodi-Singer boundary-value problem to the operator $\widetilde{D}$. Conequently, we have that the map 
\begin{align}
\widetilde{D} \oplus (\widetilde{\Pi}^- \circ r) : i \Omega^1 (X) \oplus \Gamma (S^+) &\rightarrow i \Omega_+^2 (X) \oplus \Gamma(S^-) \oplus i \Omega^0 (X) \oplus \widetilde{H}^-  \nonumber
\end{align}
is Fredholm, where $\widetilde{\Pi}^-$ is the projection onto the nonpositive eigenspace of $\widetilde{L}$, denoted by $\widetilde{H}^- \subset \mathit{Coul}(Y) \oplus i \Image{d} \oplus i \Omega^0 (Y)$.

% Consider a subspace $W^0 := \left\lbrace f \in i\Omega^0 (Y) \, | \, f|_{Y_i} = c_i \text{ and} \sum c_i = 0 \right\rbrace $ which is of dimension $b_0 (Y) - 1$.    According to (\ref{eq sumFredholm}), we want to replace $\widetilde{H}^-$ by a sum $H_0^- \oplus W$ where $W$ is given by  $i \Image{d} \oplus W^0$. We see that the kernel of $d^* \oplus (\Pi_W \circ r)$ is precisely $\mathit{Coul}^{\mathit{CC}} (X)$, where $\Pi_W$ is a projection from $i \Image{d} \oplus i \Omega^0 (Y)$ onto $W$.

Let $H_1^- \subset i \Image{d} \oplus i \Omega^0 (Y)$ be the nonpositive eigenspace of $L_1$ and $\Pi_1^-$ be its spectral projection and we see that $\widetilde{\Pi}^- = \Pi_0^- \oplus \Pi_1^-$. 
 Let $W \subset i \Omega^0 (Y)$ be the subspace of locally constant functions and let  $\Pi_{2}$ be the projection from $i \Image{d} \oplus i \Omega^0 (Y)$ onto $ i \Image{d} \oplus W^{}$ whose kernel is $\left\{ 0 \right\} \oplus \left\{ f  \, | \, \int_{Y_i} f = 0 \right\}$. We would like to apply the earlier observation (\ref{eq sumFredholm}) to the map \(\widetilde{D} \oplus ((\Pi^-_0 \oplus \Pi_2) \circ r)\) because the kernel of $d^* \oplus (\Pi_2 \circ r)$ is precisely $\mathit{Coul}^{\mathit{CC}} (X)$.
  
We start comparing \(\Image{\Pi^-_1}\) and \(\Ker{\Pi_2}\) by observing that $d d^*$ is positive on $i\Image{d}$. Consequently, the pairs $(b , d^* (d d^*)^{-1/2} b )$ and $(0,c)$  lie in $H_1^-$ for any \(b \in i\Image{d}\) and for any locally constant function $c$. Moreover, the intersection of $H^-_1$ and $\left\{ 0 \right\} \oplus \left\{ f  \, | \, \int_{Y_i} f = 0 \right\}$ is the zero set, so we can see that any element in $i \Image{d} \oplus i \Omega^0 (Y)$ can be written uniquely as the sum of elements from these two subspaces.

Consequently, the kernel of $\Pi^-_0 \oplus \Pi_{2}$ is complementary to the image of $\Pi^-_0 \oplus \Pi_1^-$. By Proposition 17.2.6 from \cite{Mono}, we can conclude that the operator 
$$\widetilde{D} \oplus ((\Pi^-_0 \oplus \Pi_{2}) \circ r) :  i \Omega^1 (X) \oplus \Gamma (S^+) \rightarrow i \Omega_+^2 (X) \oplus \Gamma(S^-) \oplus i \Omega^0 (X) \oplus (H^-_0 \oplus i \Image{d} \oplus W) $$ 
is Fredholm.  From (\ref{eq sumFredholm}), we set $D_1= \hat{D}\oplus (\Pi^-_0 \circ r )$ and $D_2 = d^* \oplus (\Pi_2 \circ r)$ and deduce that the map $\hat{D} \oplus (\Pi^-_0 \circ r )$ is Fredholm with index
\begin{align}
        \operatorname{Ind} (\hat{D} \oplus (\Pi^-_0 \circ r ) ) = \operatorname{Ind} (\widetilde{D} \oplus ((\Pi^-_0 \oplus \Pi_2) \circ r) )  + \dim{\Coker{D_2}}.
\nonumber
\end{align}

To find a formula for the index, one can observe that the operators $\widetilde{D} \oplus (\widetilde{\Pi}^- \circ r)$ and $\widetilde{D} \oplus ((\Pi^-_0 \oplus \Pi_{2}) \circ r) $ have the same index (see the proof of Proposition 17.2.6 in \cite{Mono}). From the proof of Proposition~\ref{prop 1-decomp}, one can deduce that the cokernel of $D_2$ is isomorphic to the space of constant functions on \(Y\). Hence, we obtain \(\operatorname{Ind} (\hat{D} \oplus (\Pi^-_0 \circ r ) ) = \operatorname{Ind} (\widetilde{D}  \oplus (\widetilde{\Pi}^- \circ r) ) +1 \).  

The index of $\widetilde{D} \oplus (\widetilde{\Pi}^- \circ r)$ can be computed from the index formula of the two operators $d^+ + d^*$ and $\slashed{D}^+_{A_0}$ with the spectral boundary condition. For instance, the index for $d^+ + d^*$ is given by
\begin{align}
   \operatorname{Ind}(d^+ + d^*) = -\frac{1}{2} \int_X \left( \frac{p_1 (X)}{3} + e(X)\right)  + \frac{\eta_{sign} - k_{sign}}{2}, \nonumber
  \end{align}  
where $p_1 (X)$, $e(X)$ are the Pontryagin class and the Euler class of \(X\) and $\eta_{sign}$, $k_{sign}$ are the eta invariant and the dimension of the kernel of the odd signature operator on $Y$ respectively. The kernel of the odd signature operator is the space of harmonic $0$-forms and $1$-forms of $Y$. Using the signature theorem and the Gauss-Bonnet theorem, we have
\begin{align}
  \sigma(X) + \chi(X)  = \int_X \left( \frac{p_1 (X)}{3} + e(X)\right)  - \eta_{sign}, \nonumber
  \end{align}
which gives
\begin{align}
   \operatorname{Ind}(d^+ + d^*) = - \frac{\sigma(X) + \chi(X) + b_0 (Y) + b_1(Y)}{2} . \nonumber
  \end{align}  
One can extract from the cohomology long exact sequence  of the pair $(X,Y)$ that 
\[ \sigma(X) + \chi(X) + b_0 (Y) + b_1(Y) = 2 ( b_0 (X) -  b_1(X) + b^+ (X) + b_1(Y)). \]   
Putting everything together, we have the desired quantity
\begin{align}
\operatorname{Ind} (\hat{D} \oplus (\Pi^-_0 \circ r ) ) &= 2\operatorname{Ind}_{\mathbb{C}} (\slashed{D}_{A_0}^+ ) - ( b_0 (X) -  b_1(X) + b^+ (X) + b_1(Y)) +1 \nonumber \\
&= 2\operatorname{Ind}_{\mathbb{C}} (\slashed{D}_{A_0}^+ ) +  b_1(X) - b^+ (X) - b_1(Y). \label{eq index4db}
\end{align}

Finally, there is an elliptic estimate for $\widetilde{D} \oplus (\widetilde{\Pi}^- \circ r)$ as a consequence of the Atiyah-Patodi-Singer theorem. Since $\Pi^-_0 \oplus \Pi_{2}$ is an isomorphism on the image of $\Pi^-_0 \oplus \Pi_1^-$, we also have an estimate for the operator $\widetilde{D} \oplus ((\Pi^-_0 \oplus \Pi_{2}) \circ r) $
     \begin{align*}
        \left\|  x \right\|_{L^2_{k+ 1}} \leq C \left( \left\| \widetilde{D}  x \right\|_{L^2_{k}} + \left\| ((\Pi^-_0 \oplus \Pi_{2}) \circ r) x \right\|_{L^2_{k+ 1/2}} + \left\| x \right\|_{L^2} \right).
\end{align*}      
Restricting \(x\) to $\Ker{D_2} = \mathit{Coul}^{\mathit{CC}} (X)$, we obtain the desired estimate. 
        
\end{proof}

\begin{rem} We can also replace $\Pi^-_0$ by any projection $\Pi^-$ commensurate to $\Pi^-_0$, i.e. a projection such that \(\Pi^- - \Pi^-_0\) is compact. The index will change according to the formula $\operatorname{Ind}(\hat D \oplus \Pi^- ) = \operatorname{Ind}(\hat D \oplus \Pi^-_0 ) + \operatorname{Ind}(\Pi^- \Pi^-_0)$, where $\Pi^- \Pi^-_0$ denotes the Fredholm operator   
\( \Pi^- \Pi^-_0 : \Image{ \Pi^-_0} \rightarrow \Image{\Pi^-}  \). Furthermore, the constant in the estimate (\ref{eq ellipest}) can be fixed for all such commensurate projections.
   \end{rem}

%%%%%%%%%%%%%%%%%%%%%%%%%%

\section{Finite Dimensional Approximation for the Seiberg-Witten map with boundary}
\label{sec main}

We briefly recall the construction of finite dimensional approximation on the boundary  3-manifold \(Y\). Through out the section, we work on a general setting with no restriction on \(b_{1}(Y)\) (cf. \cite{TK1}). At the very end, we will specialize to the case \(b_1(Y)=0\) (cf. \cite{Man1}).

On the Coulomb slice \(\mathit{Coul}(Y)\), we have a Seiberg-Witten vector field \(F\) given by
\begin{align}
F(b,\psi) =  \left(  *db+ \quadr{\psi} +\frac{1}{2} *F_{B_0^t }  - d \bar{\xi}(\psi) \, , \, \slashed{D}_{B_0 } \psi +b \cdot \psi - \bar{\xi}(\psi) \psi   \right), \nonumber
\end{align}
where \(\bar{\xi}(\psi)\) is a unique function that satisfies $\Delta \bar{\xi}(\psi) = d^* \left( \quadr{\psi}  + \frac{1}{2} *F_{B_0^t} \right)$ and $\int_Y \bar{\xi}(\psi) = 0$. This vector field arises from a (nonlinear) projection of the gradient of the Chern-Simons-Dirac functional onto the Coulomb slice. Note that \(F\) can also be decomposed as a sum \(F = D + Q\) where \(D = (*d , \slashed{D}_{B_0} )\) is the linear operator from (\ref{eq linearmap3d}) and the nonlinear term $Q$ has nice compactness properties. 

When necessary, we pick a perturbation $\mathfrak{q}$ on the  3-manifold $Y$. This induces a perturbation on the cylinder \(I \times Y\). On the 4-manifold \(X\), we will also pick a perturbation $\hat{\mathfrak{p}}$ supported in the collar neighborhood of \(Y\) such that the restriction to $\left\{ 0 \right\} \times Y$ is $\mathfrak{q}$. In addition, as in \cite[\S24.1]{Mono}, we assume that $\hat{\mathfrak{p}}$ is of the form
\begin{align}
        \hat{\mathfrak{p}} = \beta \mathfrak{q} + \beta_0 \mathfrak{p}_0 , \nonumber
\end{align}
where $\beta$ is a cut-off function with value 1 near the boundary, $\beta_0$ is a bump function supported in $(-C , 0)$, and $ \mathfrak{p}_0$ is another perturbation on $Y$. We will always choose \( \mathfrak{q} , \mathfrak{p}_0 \) from the Banach space of tame perturbations (c.f. \cite[\S 11]{Mono}). For the rest of the paper, it is understood that the Seiberg-Witten map and the Chern-Simons-Dirac functional are perturbed. We will continue to write \(SW = \hat D + \hat Q\) and \(F = D\ + Q\) as we keep the linear  parts the same and we add terms from perturbation to the nonlinear parts. When the perturbations are tame, the nonlinear terms \(\hat Q\) and \(Q\) retain appropriate compactness properties.

Choose a closed and bounded subset \(\mathcal{R}\) in the \(L^{2}_{k+1/2}\) completion of \(\mathit{Coul}(Y)\) with the following property: if a trajectory \(y(t)\) satisfies
 \begin{align}
        - \frac{\partial}{\partial t} y (t) =  F (y(t)) \nonumber
\end{align}  
and lies in \(\mathcal{R}\) for all time \(t\in\mathbb{R}\), then \(y(t)\) in fact lies in the interior of \(\mathcal{R}\) for all time. This subset \(\mathcal{R}\) can be viewed as an isolating neighborhood for the Seiberg-Witten flow. A key result (cf. \cite[Proposition~11]{TK1}) for constructing the Floer spectrum is that a compact subset \(\mathcal{R} \cap W\) of a sufficiently large finite dimensional subspace \(W\) is an isolating neighborhood for a \emph{compressed flow} on \(W\) given by a projected vector field    \begin{align}
        - \frac{\partial}{\partial t} y (t) =  \pi_W F (y(t)), \nonumber
\end{align}
where \(\pi_W\) is a projection onto \(W\). When \(b_1 (Y) = 0\), a large ball \(B(2R)\) in \(\mathit{Coul}(Y)\) can be chosen for such an isolating neighborhood (cf. \cite[Proposition~3]{Man1}). For the construction of \(\mathcal{R}\), we refer to \cite{ManK} for the case \(b_1 (Y) = 1\) and to \cite[\S4.4]{TK1} for a general case. 

As a result, one can obtain Conley index \(\cindex{\mathcal{R} \cap W}\) with respect to this compressed flow (see Appendix~\ref{App con} for a background in Conley index theory). This will allow us to construct the Floer spectrum later on.

Now, we consider a map
\begin{align*}
SW \oplus (\Pi^-_{} \circ r) : \mathit{Coul}^{\mathit{CC}} (X) \rightarrow i \Omega_+^2 (X) \oplus \Gamma(S^-) \oplus  H^- ,
\end{align*}
where \( \Pi^-\) is a projection onto a semi-infinite subspace \(H^-\) commensurate to \(H^-_0\) as introduced in Proposition~\ref{prop bFred} and in the subsequent remark. For convenience, we denote  the 4-dimensional part of the codomain,  $i \Omega_+^2 (X) \oplus \Gamma(S^-)$, by $\mathcal{V}_X$.

Notice that there is a residual action on \(\mathit{Coul}^{\mathit{CC}} (X)\) by the quotient
\[ \mathcal{G} / \mathcal{G}^\bot \simeq H^1(X;\mathbb{Z}) = \mathbb{Z}^{b_1 (X)}, \]
which can be viewed as a group of harmonic maps \(\mathcal{H}^1_{\mathit{CC}}(X)\) with double Coulomb condition. By Proposition~\ref{prop 1-decomp}, there is a unique map \(u\) with \(u^{-1} du \in \Omega_{\mathit{CC}}^{1}\) for each cohomology class.   Elements of \(\mathcal{H}^1_{\mathit{CC}}(X)\) span a subspace of dimension \(b_1(X)\) in the slice \(\mathit{Coul}^{\mathit{CC}} (X) \) and we will denote by \(\mathcal{U}_X\) its orthogonal complement. This gives a decomposition \(\mathit{Coul}^{\mathit{CC}} (X) \simeq \mathcal{U}_X \oplus \mathbb{R}^{b_1(X)}\) and one may view \(\mathcal{U}_X\) as a fiber of this (trivial) bundle.
\begin{note} If \(X\) is closed or the Coulomb-Neumann condition is used, the subspace \(\mathcal{U}_X\) can be identified with $\Image{d^*} \subset \Ker{d^*}$.
\end{note}

The rest of the construction will closely follow the construction of the relative Bauer-Furuta invariant in \cite{Man1}. First, we pick a sufficiently large ball $B(R)$ of \(\mathcal{U}_X\) in \(L^{2}_{k+1}\) topology. The image of this ball under the restriction map is bounded. We can pick a bounded isolating neighborhood $\mathcal{R}$ containing this image in the slice \(\mathit{Coul}(Y)\). The choices of \(B(R)\) and $\mathcal{R}$ will depend on universal constants in Corollary~\ref{cor universalb}. 

For each positive integer $n$, let \(H^{-}_n\) be a semi-infinite subspace of $\mathit{Coul}(Y)$ such that its projection \(\Pi^-_n\) is commensurate to $\Pi^-_0$. Since $\hat{D} \oplus (\Pi^-_n \circ r)$ is Fredholm, we pick a finite-dimensional subspace $V_n \oplus W_n$ of the codomain $\mathcal{V}_X \oplus H_n^-$ such that it contains the cokernel this map. We will require that \(H^-_n\), $V_n$, and $W_n$ are sequences of increasing subspaces approaching the whole spaces. In particular, we  need that $\mathcal{R} \cap W_n$ is an isolating neighborhood of the compressed Seiberg-Witten flow on $W_n$.  

For example, as in \cite{Man1}, we may choose \(W_n\) to be the subspace containing all eigenspaces of \(D\) with respect to eigenvalues in an interval \([\lambda_n , \mu_n ]\) and \(H^-_n\) to be the semi-infinite subspace containing all eigenspaces with eigenvalues in an interval \((-\infty , \mu_n ]\) with \(-\lambda_n , \mu_n \rightarrow \infty\).   

Let $U_n$ be the preimage of $V_n \oplus W_n$ under the linear map $\hat{D}\oplus (\Pi^-_n \circ r)$. Consider a projected Seiberg-Witten map
\begin{align}
\pi_{V_n \oplus W_n} \circ ( SW \oplus (\Pi^-_n \circ r) ) : U_n \rightarrow V_n \oplus W_n \nonumber 
\end{align} and this gives a map
\begin{align*}
        B(R, U_n ) \rightarrow V_{n}\times (\mathcal{R} \cap W_n )
\end{align*}
when restricting to the ball. Let \(\epsilon_n\) be a sequence of positive numbers approaching \(0\), then we will try to show that, for \(n\) sufficiently large, this induces a  map of the form 
\begin{align*}
        B(R, U_n ) / S(R , U_n) \rightarrow V_{n} / B(\epsilon_n , V_n )^C \wedge N_n / L_n,
\end{align*}
where $B(\epsilon_n , V_n )^C$ is an element in \(V_{n}\) outside the small ball and \((N_n,L_n)\) is an index pair for \(\mathcal{R} \cap W_n\) with respect to the compressed Seiberg-Witten flow on \(W_n\). Consequently, this will give a map from a sphere to a suspension of the Conley index
\begin{align}
        S^{U_n} \rightarrow S^{V_n} \wedge \cindex{\mathcal{R} \cap W_n}. \nonumber
\end{align}
A crucial part in the construction is to assure that we can find such an index pair for which the induced map is well-defined. 

There are three main ingredients to establish such maps. First, we recall results regarding the moduli space of Seiberg-Witten solutions on a 4-manifold with boundary.  

\subsection{Boundedness of X-trajectories}
Let \(X^* := X \cup_Y ( [0,\infty) \times Y)\) be a manifold with cylindrical ends. A Seiberg-Witten solution on \(X^*\), also known as an \emph{\(X\)-trajectory}, can be viewed as a pair of a solution on \(X\) and a half-trajectory on \(Y\) with compatibility condition. In particular, there is a homeomorphism between moduli spaces (cf. \cite[Lemma~24.2.2]{Mono})
\begin{align}
\mathcal{M}(X^*) \simeq \mathcal{M}(X) \times_{\mathcal{B}( Y)} \mathcal{M}([0,\infty) \times Y).
\label{eq modulidecomp} 
\end{align}   

For an \(X\)-trajectory \(\gamma\), one can define its topological energy \(\mathcal{E}(\gamma)\) (cf. \cite[\S24]{Mono}). When \(\gamma\) is asymptotic to \(\mathfrak{a}\) on the cylindrical end, we have that
\begin{align}
\mathcal{E}(\gamma) = C_X - \CSD(\mathfrak{a}),   
\label{eq Xenergy}
\end{align}
where \(\CSD\) is the (perturbed) Chern-Simons-Dirac functional on \(Y\) and \(C_X\) is a constant which depends only on \(X\), a spin\(^c\) structure, a metric, and a perturbation. We now state the compactness result in Seiberg-Witten theory.
\begin{proposition}(\cite{Mono}, Proposition 24.6.4) For \(C>0\), the space of (broken) \(X\)-trajectories with energy \( < C\) is compact in the topology defined in \cite{Mono}. \label{prop 4dcompact}
\end{proposition}

 Next, we will show that an \(X\)-trajectory with finite energy actually has its energy bounded by a universal constant. 

\begin{lemma} There is a uniform bound for an energy of any \(X\)-trajectory with finite energy. When the spin\(^c\) structure \(\mathfrak{s}\) of \(Y\) is nontorsion, we further require that the perturbation is regular in the sense that all moduli spaces of trajectories are regular.
\label{lem finiteenergy}
\end{lemma} 
\begin{proof} First, we observe that an \(X\)-trajectory with finite energy is always asymptotic to a critical point of \(\CSD\) on the cylindrical ends.
By (\ref{eq Xenergy}), we only need to consider the value of \(\CSD\) at its critical points. When \(\mathfrak{s}\) is torsion, one could see that the statement is trivial  by compactness of the solutions to the 3-dimensional Seiberg-Witten equations modulo gauge.

When \(\mathfrak{s}\) is nontorsion, we also have to control the gauge action, which can be viewed as a homotopy class of an \(X\)-trajectory on the quotient configuration space.   When the perturbation is regular, the set of critical points modulo gauge is a finite set. The following argument will be similar to the finiteness result in monopole Floer homology.

Let \([\mathfrak{a}] \) be a class of critical points and fix an \(X\)-trajectory \([\gamma_0 ]\) asymptotic to \([\mathfrak{a}]\) with homotopy class \(\theta_0\). Suppose that \([\gamma]\) is an arbitrary element of a moduli space of \(X\)-trajectories asymptotic to \([\mathfrak{a}]\) with a homotopy class \(\theta\). When this moduli space is regular, its dimension is given by a quantity \(\operatorname{gr}_{\theta}(X,[\mathfrak{a}])\) with a relation
\begin{align}
\operatorname{gr}_{\theta}(X,[\mathfrak{a}]) = \operatorname{gr}_{\theta_0}(X,[\mathfrak{a}]) + \left( [u] \cup   c_1 (\mathfrak{s}) \right)  \left[Y\right],
\label{eq gradediff}
\end{align}  
where \([u] \in H^1 (Y; \mathbb{Z})\) corresponds to a class \(\theta \theta_0^{-1}\). 

On the other hand,  the difference of energies of  \([\gamma]\) and \([\gamma_0]\) is given by
\begin{align*}
\mathcal{E}(\gamma) - \mathcal{E}(\gamma_0) &= \CSD(\mathfrak{a}) - \CSD(u \cdot \mathfrak{a}) \\
&= -2 \pi^2 \left( [u] \cup   c_1 (\mathfrak{s}) \right)  \left[Y\right].
\end{align*}
Since the dimension \(\operatorname{gr}_{\theta}(X,[\mathfrak{a}])\) is nonnegative,  the equation (\ref{eq gradediff}) implies that
\[ \mathcal{E}(\gamma) \leq \mathcal{E}(\gamma_0) + 2 \pi^2 \operatorname{gr}_{\theta_0}(X,[\mathfrak{a}]).  \]
This finishes the proof  as there are finitely many classes of critical points. 
\end{proof}

\begin{rem} For the \(Pin(2)\)-equivariant case, we will also use equivariant perturbations as in the upcoming work of F. Lin \cite{FLin}.  In this context, one considers Morse-Bott version of Floer homology, where a critical point \(\mathfrak{a}\) is replaced by a critical submanifold \(\mathfrak{C}\). There are analogous compactness and transversality results as well as the relative grading \(\operatorname{gr}_{\theta}(X,[\mathfrak{C}])\).

The argument above can be directly applied to a nonexact perturbation which is either balanced or positively monotone (cf. \cite[\S29]{Mono}).
\end{rem}

We now deduce the uniform boundedness result for \(X\)-trajectories with finite energy\footnote{This is analogous to the finite type condition in \cite{Man1}. However, the finite energy condition implies the finite type condition (cf. \cite[\S5]{Mono}).} with respect to a particular gauge fixing.  
\begin{corollary} A pair \((x,y)\) of  a solution  \(x \in \mathcal{U}_X \) with $\mathit{SW}(x)=0 $ and a half-trajectory \(y: [0, \infty) \rightarrow \mathit{Coul}(Y)\) satisfying $- \frac{\partial}{\partial t} y (t) =  F (y(t))$ and $r(x) = y(0) $ gives rise to an \(X\)-trajectory. Moreover, there are constants \(B_{k} \) and \( C_k\) such that, for any such pair \((x,y)\) with finite energy,
we have
\begin{itemize}
\item  \(\left\|  x \right\|_{L^2_{k+ 1}} \leq B_k\).
\item For each \(t \geq 0\), there is a harmonic gauge transformation \(u_t \in H^1(Y;\mathbb{Z}) \) such that \(\left\| u_t \cdot  y(t) \right\|_{L^2_{k+ 1/2}} \leq C_k\).
\end{itemize}

\label{cor universalb}
\end{corollary}

\begin{note} The gauge transformation \(u_t\) comes from a residual gauge action on \(\mathit{Coul}(Y)\). For any \(y = (\alpha, \psi) \in \mathit{Coul}(Y) \), there is a unique gauge transformation (up to constant) \(u\) such that \(\alpha - u^{-1}du\) satisfies the period condition \(\int \beta_j \wedge (\alpha - u^{-1}du) \in [0,2\pi) \), where \(\{ \beta_j \}\) is a dual basis of \(H_1 (Y;\mathbb{Z})\). This condition was used in the proof of compactness results in \cite{Mono}. 
\end{note}

\begin{proof} For the first part  we will proceed similarly to the construction of the homeomorphism (\ref{eq modulidecomp}), that is we will glue a solution on \(X\) and a half-trajectory on \(Y\) to obtain a solution on \(X^*\). Recall that $y$ is a Coulomb projection of a Seiberg-Witten trajectory $\bar y$ with $\bar y (0) = y(0) $ (cf. \cite[\S 4.2]{TK1}). We can write a solution \(x\) near the boundary of \(X\) and $\bar y$ in cylindrical coordinates as in (\ref{eq 1formcylin}) 
\begin{align*}
x &= \left(   \alpha_1(t) + \beta_1(t) + \gamma_1(t) dt , \psi_1(t) \right), \\
\bar y (t) &= \left(   \alpha_2(t) + \beta_2(t) + \gamma_2(t) dt , \psi_2(t) \right).
\end{align*}
The hypothesis $x \in \mathcal{U}_X$ and $r(x) =y$ implies that $\alpha_1(0) = \alpha_2(0)$, $\beta_1(0) = \beta_2(0) = 0$, and $\psi_1(0) = \psi_2(0)$. 

The only thing we need to concern is that $\bar y$ is in temporal gauge, particularly $\gamma_2(0) =0$, but $\gamma_1(0)$ is not necessarily zero. Let $f: [0,\infty) \rightarrow [0,\infty)$ be a smooth compactly supported function with $f'(0) = 1$ and consider a gauge transformation $U(t) = e^{\int_0^t f(s) \gamma_1(0) ds}$ on $[0,\infty) \times Y$. We see that $U^{-1} d U = \int_0^t f(s)  d_Y (\gamma_1(0)) ds + f'(t) \gamma_1(0) dt$ and $U(0) = id$, so that $x$ and $U(t) \bar y(t)$ now agree on the boundary $\{ 0\} \times Y$. Note that the resulting solution on \(X^*\) is in a mixed gauge condition: the part on $X$ is in the double Coulomb gauge and the part on $[0,\infty) \times Y$ is in temporal gauge away from the boundary. One can always turn a  solution on $X^*$ to a solution in this mixed gauge in a unique way.  

For the second part, Proposition~\ref{prop 4dcompact} combined with Lemma~\ref{lem finiteenergy} gives universal bounds on the Sobolev norms of \(X\)-trajectories with finite energy in the above mixed gauge.  Since the restriction and Coulomb projection are continuous,  we have universal bounds for the Sobolev norms of a pair $(x,y)$ described above as desired.   
\end{proof}

\subsection{Approximated solutions}
Secondly, we will need convergence result for approximated \(X\)-trajectories. The idea is to combine finite dimensional approximation arguments for both closed 4-manifolds and closed 3-manifolds. 

For the rest of the section, we let $\left\{ \hat{\pi}_n \right\}$ be a sequence of projection onto a finite-dimensional subspace $V_n$\ of \(\mathcal{V}_X\) such that \(\hat{\pi}_n \rightarrow id\) strongly. Similarly, let \(\left\{ \Pi^-_n \right\}\) be a sequence of projection onto a semi-infinite subspace \(H^-_{n}\) of \(\mathit{Coul}(Y)\) with the following properties;
\begin{itemize}
\item \(H^-_{n}\) contains the nonpositive eigenspace \(H^-_0\) of \(D\),
\item \( \Pi^-_n\) is commensurate to \( \Pi^-_0\), the projection onto \( H^-_0\),
\item  \(\Pi^-_n \rightarrow id\) strongly.
\end{itemize}
We also let \(\left\{ \pi_n \right\}\) be a sequence of projection onto a finite-dimensional subspace \(W_{n}\) of \(H^-_{n}\) such that \(\pi_n \rightarrow id\) strongly. In addition, we require that the commutator $D \pi_n - \pi_n D \rightarrow 0$ in operator norm.

The following proof will be almost the same as that of Proposition 6 in \cite{Man1}, except that there is a slight simplification because we do not need to consider a nonlinear Coulomb projection in the argument.  
 
\begin{lemma} \label{lemma 4dhalflimit} Let $\left\{ x_n \right\}$ be a bounded sequence in the $L^2_{k+1}$ completion of \(\mathcal{U}_X\) such that $\hat{D} x_n \in  V_n$ and \((\Pi^-_n \circ r)x_n \in W_n\). Suppose that $ (\hat{D} + \hat{\pi}_n \hat{Q} ) x_n \rightarrow 0$ in $L^2_{k}$ and  there is a sequence of half-trajectories $y_n : [0,\infty) \rightarrow W_n $ uniformly bounded in the $L^2_{k + 1/2}$ completion of \(\mathit{Coul}(Y)\) such that
\begin{align*}
        - \frac{\partial}{\partial t} y_n (t) = \pi_n F y_n(t) ,
\end{align*}
together with $y_n (0) = (\Pi^-_n \circ r) x_n$. Then, after passing to a subsequence, the sequence $\left\{ x_n \right\}$ converges to a Seiberg-Witten solution $x$ in $L^2_{k+1}$ and there exists a Seiberg-Witten half-trajectory $y $ with $y(0) = r(x)$ and \(y_n(t)\) converges to \(y(t)\) in $L^2_{k+ 1/2}$ for all \(t \ge 0\).
\end{lemma}

\begin{proof} Since $\left\{ x_n \right\}$ is bounded, there is a subsequence of $x_n$ converges to $x$ weakly in the $L^2_{k+1}$ norm. After passing to this subsequence, we have strong convergence $x_n \rightarrow x$ in $L^2_{k}$ by Rellich lemma. Since a linear map preserves weak limits and $\hat{Q}$ is continuous in $L^2_k$, we also see that $(\hat{D} + \hat{Q} ) x_n $ converges to $(\hat{D} + \hat{Q} ) x $ weakly in the $L^2_k$ norm.  

On the other hand, we have 
\begin{align*}
        \left\| (\hat{D} + \hat{Q} ) x_n \right\|_{L^2_k} \leq \left\| (\hat{D} + \hat{\pi}_n \hat{Q} ) x_n \right\|_{L^2_k} + \left\| (1- \hat{\pi}_n ) \hat{Q} x_n \right\|_{L^2_k} .
\end{align*}
The first term goes to $0$ by the hypothesis while the second term also goes to $0$ because $(1- \hat{\pi}_n )$ converges to $0$ uniformly on a compact set (the image of a bounded set under  \(\hat{Q}\)). Hence, $(\hat{D} + \hat{Q} ) x $ must be equal to $0$. Moreover, we have
\begin{align}
        \left\| \hat{D}(x_n - x) \right\|_{L^2_k} \leq \left\| (\hat{D} + \hat{Q} ) x_n \right\|_{L^2_k} + \left\| \hat{Q} x - \hat{Q} x_n \right\|_{L^2_k} \rightarrow 0 .
\label{eq converg1}
\end{align}

Next, we move on to the 3-dimensional part. Similar to the proof of Proposition 3 in \cite{Man1}, there is a half-trajectory $y : [0,\infty) \rightarrow \mathit{Coul}(Y) $ such that $y_n (t) \rightarrow y(t)$ in $L^2_{k+ 1/2}$ uniformly on any compact subset of the open half-line $(0,\infty) $ but the convergence holds only in $L^2_{k- 1/2}$  on a compact subset of the closed half-line $[0,\infty)$. In addition, \(y\) is a Seiberg-Witten trajectory  
\begin{align*}
        - \frac{\partial}{\partial t} y (t) =  F y(t) .
\end{align*}

Applying the fundamental theorem of calculus to  $  e^{t D} \Pi^-_0 \gamma(t)$, we have 
\begin{align}
        e^D \Pi^-_0 \gamma(1) - \Pi^-_0 \gamma (0) &=  \int_0^1 e^{t D} \Pi^-_0 \left( D \gamma(t) + \frac{\partial}{\partial t} \gamma(t) \right) dt, \nonumber
\end{align}
where we use the fact that \(D\) and \(\Pi^-_0\) commute.
We will consider the integrand when $\gamma = y - y_n$ and use a decomposition 
\begin{align}
        (D + \frac{\partial}{\partial t} ) (y_n - y) = ( D \pi_n - \pi_n D ) y_n + \pi_n (Q y - Q y_n ) + (1 - \pi_n ) Q y .
\label{eq converg2}        
\end{align}
For the last two terms, we use the fact that $e^{ D} \Pi^-_0$ and $Q$ are bounded maps on $L^2_{k+ 1/2}$, so that, for some constant $R_0$,
\begin{align}
\left\| e^{t D} \Pi^-_0 \left( \pi_n (Q y (t) - Q y_n (t)) + (1 - \pi_n ) Q y(t) \right) \right\|_{L^2_{k+ 1/2}} \leq R_0 , \label{eq R0ofQ}
\end{align} 
uniformly on $t \in [0,1]$. Note that $e^{ D} \Pi^-_0$ is bounded because we consider the exponential of \(D\) restricted to its negative eigenspace.   

Let us fix $\delta > 0$. By continuity of $Q$, we have that $Q y_n (t) \rightarrow Q y(t)$ in $L^2_{k+ 1/2}$ uniformly on $[\delta , 1]$. Moreover, $\left\| y(t) \right\|_{L^2_{k+ 1/2}}$ is uniformly bounded on this interval. By compactness of $Q$, we can conclude that $(1 - \pi_n ) Q y(t) \rightarrow 0$ in $L^2_{k+ 1/2}$ uniformly on $[\delta , 1]$ as well. As a result, for any \(\epsilon >0\), we can find a sufficiently large integer \(N_0\) depending on a fixed \(\delta_0 < \epsilon/2R_0\) so that the integral
\begin{align*}
         \int_{\delta_0}^1 \left\| e^{t D} \Pi^-_0 \left( \pi_n (Q y (t) - Q y_n (t)) + (1 - \pi_n ) Q y(t) \right) \right\|_{L^2_{k+ 1/2}} dt < \epsilon/2
\end{align*}
when \(n > N_0\). Using (\ref{eq R0ofQ}), we add the integral on \([0,\delta_0]\) and obtain, for \(n > N_0\), 
\begin{align*}
        \int_{0}^1 \left\| e^{t D} \Pi^-_0 \left( \pi_n (Q y (t) - Q y_n (t)) + (1 - \pi_n ) Q y(t) \right) \right\|_{L^2_{k+ 1/2}} dt < \delta_0 R_0 + \epsilon/2 < \epsilon .
        \end{align*} 

For the first term on the right hand side of (\ref{eq converg2}), we use the hypothesis that the commutator $D \pi_n - \pi_n D \rightarrow 0$ as a bounded operator on $L^2_{k+ 1/2}$. Since $\{ y_n \}$ is uniformly bounded, we see that $( D \pi_n - \pi_n D ) y_n (t) \rightarrow 0$ in $L^2_{k+ 1/2}$ uniformly on $[0, \infty )$.
Putting everything together, we have
\begin{align*}
        \left\| \Pi^-_0 (y(0) - y_n (0)) \right\| &\leq \left\| e^D \Pi^-_0 (y(1) - y_n (1)) - \Pi^-_0(y(0) - y_n(0)) \right\| + \left\| e^D \Pi^-_0 (y(1) - y_n (1)) \right\| \\
        &\leq  \int_0^1  \left\| e^{t D} \Pi^-_0 (D + \frac{\partial}{\partial t} ) (y_n (t) - y(t)) \right\| dt + \left\| e^D \Pi^-_0 (y(1) - y_n (1)) \right\|  
\end{align*}
and we can conclude that \(\Pi^-_0 y_n (0) \rightarrow \Pi^-_0 y (0) \) in $L^2_{k+ 1/2}$ topology.  

Since $\Pi^-_n  r( x_n) = y_n (0) $ and \(H^-_0 \subset H^-_n\), we have \(\Pi^-_0 r (x_n) = \Pi^-_0 y_n (0)  \). On the other hand, we see that $\Pi^-_0 r (x_n)$ converges to $\Pi^-_0 r(x)$ weakly in $L^2_{k+ 1/2}$ because $x_n$ converges to $x$ weakly in the $L^2_{k+1}$ and \(\Pi^-_0 r\) is bounded linear. Thus we must have $\Pi^-_0 r (x) = \Pi^-_0 y(0)$. Together with (\ref{eq converg1}), the elliptic estimate (\ref{eq ellipest}) implies that $x_n$ converges to $x$ in $L^2_{k+1}$.

Since \(r\) is bounded linear, we also have that \(r(x_{n })\) converges to \(r(x)\) in $L^2_{k+ 1/2}$. Then, we observe that
\begin{align*}\left\Vert y_n (0) - r(x)\right\Vert = \left\Vert \Pi^-_n  r( x_n) - r(x)\right\Vert \leq \left\Vert \Pi^-_n  \left(r( x_n) - r(x)\right)\right\Vert + \left\Vert \left( 1 - \Pi^-_n   \right)r(x)\right\Vert   \end{align*}
so that \(y_n(0)\) converges to \(r(x) \) in $L^2_{k+ 1/2}$ because \(\Pi^-_n \) converges to the identity strongly.  Since \(y_n(0)\) converges to \(y(0)\) in $L^2_{k- 1/2}$, we must also have $ r (x) = y(0)$ and the convergence \(y_n(0) \rightarrow y(0)\) in $L^2_{k+ 1/2}$.

\end{proof}

\subsection{The main results}
The last ingredient is a technical lemma from Conley index theory to guarantee existence of an appropriate index pair. Recall that we want a map of the form
\begin{align*}
        B(R, U_n ) / S(R , U_n) \rightarrow V_{n} / B(\epsilon_n , V_n )^C \wedge N_n / L_n .
\end{align*}
The situation is almost the same as the set up of Proposition~\ref{prop maptocon} in Appendix~\ref{App con} except that there is a map in the first factor. Since we are collapsing everything outside of the ball in \(V_n \), we can focus only on the second factor by considering 
\begin{align}
\Pi^-_n \circ r: B(R, U_n ) \cap  (\hat{\pi}_n SW)^{-1}( B(\epsilon_n , V_n )) \rightarrow \mathcal{R} \cap W_n .
\label{eq preSW1} 
\end{align}
Note that the image of \(B(R, U_n )\) under \(\Pi^-_n \circ r\) is already in \(W_n\) by the choice of \(U_n\). Then, it is left verify to that this map  satisfies hypothesis of Proposition~\ref{prop maptocon} with $A =B(R, U_n ) \cap  (\hat{\pi}_n SW)^{-1}( B(\epsilon_n , V_n ))  $ and \(B = S(R, U_n ) \cap  (\hat{\pi}_n SW)^{-1}( B(\epsilon_n , V_n ))\).  

We now state the main result.

\begin{proposition} Let us fix a sufficiently large radius $R$ and a sufficiently large isolating neighborhood $\mathcal{R}$ depending on the radius $R$. With the above notation, for $n$ sufficiently large, we obtain a map
\begin{align}
        S^{U_n} \rightarrow S^{V_n} \wedge \cindex{\mathcal{R} \cap W_n}
        \label{eq finitemap}
\end{align}
induced from the map \(\pi_{V_n \oplus W_n} \circ ( SW \oplus (\Pi^-_n \circ r) ): B(R, U_n ) \rightarrow V_{n}\times (\mathcal{R} \cap W_n ) \).
\end{proposition}

\begin{proof} We will prove by contradiction. After passing to a subsequence, suppose that there is a sequence of $V_n , W_n$, and \(\epsilon_n\) such that the image of the map (\ref{eq preSW1}) does not satisfy hypothesis of Lemma~\ref{relCon}. This gives a sequence  \(x_n \in B(R, U_n ) \cap  (\hat{\pi}_n SW)^{-1}( B(\epsilon_n , V_n ))\) with an image \(\Pi^-_n \circ r (x_n)\) lies in the invariant set of the compressed flow on \(\mathcal{R} \cap W_n\) in positive time direction. In other words, we have a sequence of approximated half-trajectories \(y_{n} : [0,\infty) \rightarrow \mathcal{R} \cap W_n  \) with \(- \frac{\partial}{\partial t} y_n (t) = \pi_n F y_n(t)\) and \(y_n(0) = \Pi^-_n \circ r (x_n)\). 

We now arrive at the set up to apply Lemma \ref{lemma 4dhalflimit}. As a result, the sequence  \(\{x_n\}\) converges to a 4-dimensional solution \(x\) and  \(\{y_n\}\) converges to a Seiberg-Witten half-trajectory \(y\) with \(r(x) = y(0)\). Together, we have an \(X\)-trajectory with finite energy and universal constants as in Corollary~\ref{cor universalb}. There are two cases to consider.

Case 1: \(x_n \in S(R,U_n)\). Here, we choose $R$ larger than the universal constant $B_k$. From Corollary~\ref{cor universalb}, this is a contradiction since we have an $X$-trajectory with $ \left\|  x \right\|_{L^2_{k+ 1}} = R > B_k$. 

Case 2: There exist \(t_n \geq 0\) such that \(y_n(t_n) \in \partial \mathcal{R} \cap W_n \). Here, we choose an isolating neighborhood \(\mathcal{R}\) arising from transverse cut-off of a union of balls $\mathbb{Z}^{b_1 (Y)} \cdot B(R' , \mathit{Coul}(Y))$ in the $L^2_{k+1}$ norm (cf. \cite{TK1} and \cite{ManK}) with $R'$ larger than the universal constant $C_k$. The limit $y(t)$ is asymptotic to a critical point \(\mathfrak{a}\) on the cylindrical end with $\mathfrak{a} \in \Inter{\mathcal{R}}$. This implies that \(t_{n} \rightarrow t_0 \geq 0\), so that $y(t_0) \in \partial \mathcal{R}$. This is a contradiction as \(\left\| u_{t_0} \cdot  y(t_0 ) \right\|_{L^2_{k+ 1/2}} = R' > C_k\).

\end{proof}

Let us try to keep track of choices made in the construction. The choice of $\epsilon_n$ does not matter as long as it is sufficiently small. From Proposition~\ref{prop maptocon}, the map is independent of the choice of index pairs. After passing to stable maps, the map is also independent of the choice of $V_n $.

For simplicity, we will specialize to the case when \(W_n = V^{\mu_n}_{\lambda_n}\) the sum of eigenspaces of \(D\) with respect to eigenvalues in an interval \([\lambda_n , \mu_n ]\) and \(H^-_n = V^{\mu_n}_{\infty}\) defined similarly. One can show that there is an isomorphism between Conley indices
\begin{align}
\Sigma^{-V^0_{\lambda_n}} \cindex{\mathcal{R} \cap V^{\mu_n}_{\lambda_n}} \simeq \Sigma^{-V^0_{\lambda_{n+1}}} \cindex{\mathcal{R} \cap V^{\mu_{n+1}}_{\lambda_{n+1}}}. \nonumber
\end{align}
Consequently, we can desuspend the Conley index on the right hand side of (\ref{eq finitemap}) by the corresponding negative eigenspace as above so that the resulting object, denoted by \(E(\mathcal{R})\) does not depend on the choice of \(V^{\mu_n}_{\lambda_n}\). Applying the index formula (\ref{eq index4db}) to (\ref{eq finitemap}), we obtain
\begin{align}
\mathbf{S}^{(- b^+ (X) - b_1(Y))\mathbb{R} + \operatorname{Ind}_{\mathbb{C}} (\slashed{D}_{A_0}^+ )\mathbb{C}  } \rightarrow E(\mathcal{R}),
\label{eq prefinal}
\end{align}   
where we note that \(\mathcal{U}_X\) is a subspace of \(\mathit{Coul}^{CC} (X)\) of codimension $b_1(X)$.
 
For the rest of the section, let us consider the case when $b_1 (Y) = 0$. 
\begin{proof}{(of Theorem~\ref{thm b1=0})}
The same argument in \cite{Man1} shows that different choices of sufficiently large radii \(R\) and sufficiently large isolating neighborhoods \(\mathcal{R}\) (which can be chosen to be the balls \(B(2R)\) in $\mathit{Coul}(Y)$) give maps in the same stable homotopy class. Note that, in this \(b_1(Y)=0\) case, we do not require the perturbation to be regular, so we can choose any  \( \mathfrak{q} , \mathfrak{p}_0 \) from the Banach space of tame perturbations together with any suitable bump functions \(\beta_0 , \beta\).
 Consequently, the choice of connections, metrics, and perturbations does not matter because the spaces of these choices are all contractible, except that we need to desuspend \(E(B(2R))\) again by $\operatorname{Ind}_{\mathbb{C}} (\slashed{D}_{A_0}^+ ) + \frac{\sigma(X) - c_1 (det S^+)^2}{8}$ complex dimension to obtain the Floer spectrum \(SWF(Y)\). 

Putting everything together, we obtain (\ref{eq mainb1=0}) from (\ref{eq prefinal}). Moreover, we obtain (\ref{eq mainth}) by considering a family of the above maps parametrized by the Picard torus of \(X\). 
\end{proof}

%%%%%%%%%%%%%%%%%%%%%%

\appendix
\section{Maps to Conley indices} %\cite{Kur1}, \cite{Kur2}
\label{App con}

In this appendix, we will briefly recall essential parts of Conley index theory. A thorough treatment can be found in \cite{Conley} and \cite{SalaCon}.

Let \(\phi\) be a flow on a finite-dimensional manifold \(M\) (or more generally, a locally compact Hausdorff topological space). Denote the flow action by \(\phi(x,t)\) or \(x \cdot t\) for \(x\in M\) and \(t \in \mathbb{R}\).

\begin{definition} Let $X$ be a subset of $M$.
\begin{enumerate}
  \item The invariant subset in positive direction is given by $A^+(X) := \left\{ x \in X | x \cdot \mathbb{R}^+ \subset X \right\}$.
  
        \item The \emph{maximal invariant subset} of $X$ is given by $\Inv{X} = \left\{ x \in X | x \cdot \mathbb{R} \subset X \right\}$. 
        
        \item A compact subset $X$ of $M$ is called an \emph{isolating neighborhood}  if $\Inv{X}$ is contained in $\Inter{X}$ the interior of $X$.
        \item A compact subset $S$ of $M$ is called an \emph{isolated invariant set} if there is an isolating neighborhood $X$ so that $\Inv{X} = S$.
\end{enumerate}

%Note that the term isolating neighborhood refers to that it is isolating an isolated invariant set (which is, in fact, its maximal invariant set).
\end{definition}

Given an isolated invariant set or an isolating neighborhood, one will be able to extract some topological data called Conley index, which can be viewed as a generalization of a Morse index. Now, we introduce the important concept of an index pair.

\begin{definition} \label{def indexp1}
Let $S$ be an isolated invariant set. A pair of compact subsets $(N,L)$ is called an \emph{index pair} for $S$ if the following conditions hold
\begin{enumerate}
        \item $S \subset \Inter{\operatorname{cl}(N \backslash L )} $ and $S = \Inv{\operatorname{cl}(N \backslash L ) }$,
        \item $L$ is positively invariant relative to $N$,
i.e. the condition $x \in L$ and $x \cdot [0,t] \subset N$ implies $x \cdot [0,t] \subset L$.
        \item $L$ is an exit set for $N$, i.e. if $x \in N $ but $x \cdot [0, \infty)  \nsubseteq N$, then there exists $t>0$ such that $x \cdot [0,t] \subset N$ and $x \cdot t \in L$.
\end{enumerate}
For an isolating neighborhood $X$ with $\Inv{X} = S$, we will also call $(N,L)$ an index pair for $X$ if it is an index pair for $S$.
\end{definition}

Fundamental results in Conley index theory state that an index pair always exists and that all such pairs are homotopy equivalent. As a result, one may view the Conley index as an invariant which assigns a homotopy type of such index pairs to an isolated invariant set. However, it is also important to consider the Conley index as a collection of all index pairs and natural homotopy equivalences between them. One motivation for this is to reduce ambiguity of the choice of index pairs in various constructions.

\begin{definition} For an isolated invariant set (or an isolating neighborhood) $S$, we define its \emph{Conley index} $\cindex{S}$ as a collection of objects consisting of pointed spaces $ (N/L , [L])$ arising from an index pairs $(N,L)$ for  $S$. For a pair of two index pairs, we also have a collection of \emph{flow maps} induced from the flow. These flow maps are homotopy equivalences and are naturally homotopic to each other. 
Such a collection of spaces and maps between them is also known as a connected simple system. See \cite{Kur2} or \cite{SalaCon} for the details.

\end{definition} 

In this paper, we will also need to construct maps from spaces to Conley indices. Under certain hypothesis, a map from a space to an isolating neighborhood can give rise to a map to an index pair. We begin with a lemma shown in Appendix of \cite{Man1}.
\begin{lemma} \label{relCon} Let $X$ be an isolating neighborhood with $\Inv{X} = S$. If a pair $(A,B)$ of compact subsets of $X$ satisfies the following
\begin{enumerate}
        \item If $x \in A^+(X) \cap A$, then $[0,\infty) \cdot x \cap \partial X = \emptyset $,
        \item   $ B \cap A^+(X) = \emptyset$, 
\end{enumerate}
then there exists an index pair $(N,L)$ of $S$ such that $A \subset N \subset X$ and $B \subset L$.  

\end{lemma}

%Note that $B$ is not necessarily a subset of $A$. However, all the cases we concern have the inclusion $B \subset A$.

Now, suppose that we have a map $f : A \rightarrow X$ and a compact subset $B$ of $A$. If the pair $(f(A) , f(B))$ satisfies the hypothesis of the above lemma, there exists an index pair $(N,L)$ containing \((f(A),f(B))\) and we obtain a map  $f : A/B \rightarrow N/L $    induced from the inclusion.

It remains to show that this map is independent (up to homotopy) of the choice of index pairs so that it gives a well-defined map from $A/B$ to the Conley index \(\cindex{X}\). 

Given two index pairs $(N_1 , L_1)$ and $(N_2 , L_2)$ with $f(A) \subset N_1 \cap N_2$ and $f(B) \subset L_1 \cap L_2$, we wish to show that the diagram below commutes up to homotopy

\begin{center}
\begin{tikzcd}
A/B \arrow[hook]{r}{f_1}[swap]{}
\arrow[hook]{rd}{}[swap]{f_2}
& N_1 / L_1 \arrow{d}{F}[swap]{} \\
& N_2 / L_2
\end{tikzcd}
\end{center}
where $f_1 , f_2$ are maps induced by inclusions and $F$ is a flow map from $(N_1 , L_1)$ to $(N_2 , L_2)$. 
%Since the flow maps are homotopy inverse, we can freely switch the role of $(N_1 , L_1)$ and $(N_2 , L_2)$.     
We point out that this is straightforward when  $(N_1 , L_1) \subset (N_2 , L_2)$ because the inclusion $N_1 / L_1 \hookrightarrow N_2 / L_2$ is homotopic to a flow map $N_1 / L_1 \stackrel{F}{\rightarrow} N_2 / L_2$ (cf. \cite[Proposition 3.1]{Kur2}).    

For a general case, we will construct a sequence of inclusions that relates $(N_1 , L_1)$ and $(N_2 , L_2)$ through index pairs which contain $(f(A),f(B))$. Since the subsets \(N_{i}\) and \(L_{i}\) are contained in $X$, we will consider a pair $(N_i \cup P(L_i , X) , P(L_i , X) )$ where $P(L_{i},X) := \left\{ y \cdot t \, | \, y \in Y \text{and } y \cdot [0,t] \subset X \right\}$ is the \emph{minimal positively invariant} set of $L_{i}$ relative to $X$.
It is not hard to see that these are index pairs. In addition, the subsets $N_i \cup P(L_i , X)$ and $P(L_i , X)$ are positively invariant relative to \(X\).  

Furthermore, we claim that the intersection $\bigcap_{i = 1,2} (N_i \cup P(L_i , X ) , P(L_i , X) )$ is also an index pair. Let us suppose that $x \in \bigcap_{i = 1,2} N_i \cup P(L_i , X ) $ and $x \cdot \mathbb{R}^+  \nsubseteq X$. By the exit set property of the pair \((N_i \cup P(L_i , X) , P(L_i , X) )\), there exists $t_i$ such that $x \cdot [0,t_i] \subset N_i \cup P(L_i , X)$ and $x \cdot t_i \in \ P(L_i , X)$ for $i = 1,2$. Without loss of generality, we may assume that $t_1 \geq t_2$. Since the subsets $N_2 \cup P(L_2 , X)$ and $P(L_2 , X)$ are positively invariant relative to $X$, we see that $x \cdot [0,t_1] \subset N_2 \cup P(L_2 , X)$ and $x \cdot t_1 \in P(L_2 , X)$ as well. This implies that \(\bigcap_{i=1,2} P(L_i , X )\) is an exit set for \(\bigcap_{i =1,2} N_i \cup P(L_i , X )\). It is straightforward to check other properties and verify that the intersection $\bigcap_{i = 1,2} (N_i \cup P(L_i , X ) , P(L_i , X) )$ is an index pair.
Note that, in general, the intersection of two index pairs needs not be an index pair. 

We now have a sequence of inclusions of index pairs containing $(f(A),f(B))$. This is shown in the diagram below (we abbreviate $P(L)$ for $P(L , X)$ in the diagram). 
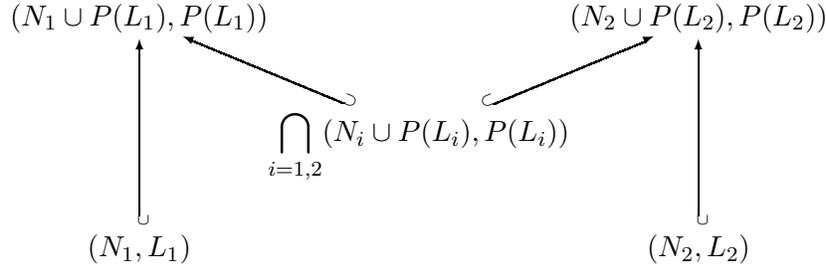
\vspace{-0.05cm}
\begin{diagram} 
(N_1 \cup P(L_1 ) , P(L_1 ) ) & & & & (N_2 \cup P(L_2 ) , P(L_2 ) ) \\
& \luInto & & \ruInto & \\
\uInto & & \bigcap_{i = 1,2} (N_i \cup P(L_i ) , P(L_i ) ) & & \uInto \\
& & & & \\
(N_1 , L_1) & & & & (N_2, L_2)
\end{diagram}

% \begin{center}
% \begin{tikzcd}
% A &  & B \\
%  & C \arrow[hook]{lu} \arrow[hook]{ru} & \\
% D \arrow[hook]{uu}  & & E \arrow[hook]{uu}
% \end{tikzcd}
% \end{center}
% 

From the above discussion, we can conclude
\begin{proposition} Let $X$ be an isolating neighborhood and \(B \subset A\) be compact sets. Suppose that there is a map \(f : A \rightarrow X\) such that a pair \((f(A),f(B))\) satisfies the hypothesis of Lemma~\ref{relCon}. Then, we have a well-defined map \(f : A/B \rightarrow \cindex{X}\) induced from the inclusion.
\label{prop maptocon}  
\end{proposition}

%\begin{singlespace}
\bibliography{research}
\bibliographystyle{plain}
%\end{singlespace}

%\bibliography{research}{}
%\bibliographystyle{amsplain}

\end{document}